\documentclass[10pt]{elsarticle}

\makeatletter
\def\ps@pprintTitle{%
   \let\@oddhead\@empty
   \let\@evenhead\@empty
   \let\@oddfoot\@empty
   \let\@evenfoot\@oddfoot
}
\makeatother

\usepackage{epsfig}
\usepackage{amssymb}

\usepackage{graphicx}
\graphicspath{ {./images/} }

\usepackage{amsmath}
\usepackage{bm}
\usepackage{mwe}
\usepackage{amsthm}

\newtheorem{theorem}{\bf Theorem}[section]
\newtheorem{lemma}{\bf Lemma}[section]
\newtheorem{remark}{\bf Remark}[section]

\def\bba{\begin{align}}
\def\bea{\end{align}}
\def\ba{\begin{align*}}
\def\ea{\end{align*}}

\def\ee{\end{eqnarray*}}
\def\be{\begin{eqnarray*}}
\def\bee{\end{eqnarray}}
\def\bbe{\begin{eqnarray}}

  \def\R{\mathrm{Re}}
   \def\I{\mathrm{Im}}
   \def\i{\mathrm i}
   \def\p{\partial}
   \def\e{\mathrm e}

   \def\n{\bm n}
    \def\e{\mathrm e}

    \def\p{\partial}

    \def\i{\mathrm i}

    \def\f{\bm f}
    \def\u{\bm u}
    \def\T{\mathbb{T}}
    \def\RR{\mathbb{R}}
    \def\Z{\mathbb{Z}}

\def\x{\bm X}

\journal{}

\title{Instability of  the Kolmogorov flow in a wall-bounded domain}

\author{Zhi-Min Chen}
\address{School  of Mathematics and Statistics, Shenzhen University,  Shenzhen 518060,  China}
\date{}

\begin{document}

\begin{abstract}
In the   magnetohydrodynamics (MHD) experiment performed by Bondarenko and his co-workers in 1979, the Kolmogorov flow loses stability  and transits  into a  secondary  steady state flow at the Reynolds number $R=O(10^3)$. This problem is modelled  as a MHD flow bounded between lateral walls under slip wall boundary condition. The existence of the secondary steady state flow is now proved.  The theoretical solution has a very good agreement with the flow measured in laboratory experiment at $R=O(10^3)$. Further transition of the secondary flow is observed numerically. Especially, well developed turbulence arises at $R=O(10^4)$.

\end{abstract}

\maketitle

\section{Introduction}

The instability of  a basic flow  has been a  principal driver in numerical  and experimental fluid dynamical studies  since the  Reynolds  pipe flow experiment \cite{1} performed  in 1883. Recently,
the instability examination  has also received significant attention in the field of mathematical fluid mechanics due to pure mathematical investigations (see, for example,  Bedrossian {\it et al.} \cite{2}, Li {\it et al.} \cite{3},  Wei and Zhang \cite{4}) for some idealized basic flows without involving  boundary layers.
Actually,   one of the  best known examples  in such a idealized flow family is probably the Kolmogorov flow
\bbe\u_0 =(\cos y,0) \, \, \mbox{ or } \,\, (\sin y,0), \label{new1}
\bee
a unidirectional  steady state solution  of the two-dimensional incompressible Navier-Stokes equations under spacially periodic boundary condition.

 This flow was introduced by Kolmogorov (see Arnold and Meshalkin \cite{5}) by suggesting the study on such a simple fluid motion  to understand the transition of Navier-Stokes flows in accordance with the Reynolds number. It was proved by Meshalkin and Sinai \cite{6} that $\u_0$ in the domain $\T\times \T$ for $\T=\RR/(2\pi\Z)$ is linearly  stabile for all $R>0$.
 Iudovich \cite{7} considered bifurcation analysis and linear spectral analysis of (\ref{new1}) in spatially periodic domains $(\frac1{k_x}\T)\times \T$ for $0<k_x <1$ and  derived  the critical Reynolds number $R_c =\sqrt{2}$ for
 (\ref{new1}) in the domain $\RR\times \T$. The numerical approximation of the  bifurcating steady-state solution of \cite{7} was given by Belotserkovskii \cite{N8}. On the other hand, there is a large literature showing Kolmogorov flow in laboratory experiments (see    Batchaev \cite{N9}, Batchaev and Dowzhenko \cite{N10}, Burgess \cite{N12},  Kolesnikov \cite{N13,N14}, Obukhov \cite{N15},  Tithof et al. \cite{N16},  ).

 Especially, in an  MHD laboratory experiment given by
  Bondarenko {\it et al.} \cite{N11},
 a thin layer of electrolyte was placed in  a plane horizontal rectangular  cell bottomed with magnetoelastic rubber, which is served as a magnetic field source and produces a sinusoidal  magnetic field
  \bbe H= H_0\sin py \bee
   perpendicular to the bottom surface of the cell. Here the amplitude strength $H_0=200 \,Oe$ and the magnetic wave number $p=2\pi/(4.4 cm)$.  An electric current  passes  transversally through the electrolyte  from  electrodes mounted on the longitudinal side walls of the cell. The motion of the electrolytic fluid  is driven by the electromagnetic Lorenz force
 \bbe  \f = (\gamma \sin py,0),\,\, \gamma = \frac1{\rho c}jH_0,\bee
 where  $\rho$ is  the density of the fluid, $c$ is the electrodynamic constant and $j$ is the electric current density.

   This three-dimensional  problem   is approximated by   the motion of an infinitesimally thin   electrolytic fluid   by ignoring the vertical motion. The effect of the bottom boundary layer reduces to effective deceleration of the horizontal flow $\u$  in accordance with a linear law
   \bbe \nu \frac{\p^2 \u}{\p z^2}= -\mu \u \bee
  on the free surface of the fluid layer,  where $\nu$ is the kinematic viscosity and $\mu$ is a friction coefficient inversely proportional to the square of the  fluid thickness.

    Thus the dynamic equations for the horizontal current on the free surface of the  electrolytic layer is reduced to the extended two-dimensional incompressible Navier-Stokes equations \cite{N11}
  \bbe
\p_t \u +\u\cdot \nabla \u + \frac1\rho \nabla P -\nu \Delta \u+\mu  \u  =\gamma (\cos py,0),\,\,\, \nabla \cdot \u=0. \label{new20}
\bee
Defining the  Reynolds number
\bbe R= \frac{\gamma}{p^3\nu^2}\bee
and the friction  number
\bbe \lambda = \frac\mu{p^2 h} \bee
 for  controlling Hartmann layer friction in MHD, equation (\ref{new20}) is transformed into the   dimensionless vorticity equation of stream function $\psi$:
\bbe\label{new2}
-\p_t \Delta \psi+ \p_x \psi \,\p_y\Delta \psi-\p_y\psi\,\p_x\Delta \psi+\frac1 R \Delta^2 \psi - \frac\lambda R\Delta \psi =\frac1 R  \sin y,
\bee
derived by   the  dimensional analysis  based on the typical length scale $p^{-1}$, the velocity scale $p^{-2} \nu^{-1}\gamma$ and the time scale $\gamma^{-1} p\nu$.

The MHD experiment showed  the much higher critical Reynolds number$R_c = O(10^3)$ for $\lambda =20$ in contrast to $R_c=O(1)$ of the idealized two-dimensional flow (see Iudovich \cite{7}). This  discrepancy was elucidated (see Sommeria and Moreau \cite{N17}, Thess \cite{N18,N19},    Bondarenko {\it et al.} \cite{N11}) by taking into account high Hartmann layer friction as only the region   $\lambda >>1$ is accessible to  MHD experimental investigations.


The  analytical results  of Iudovich \cite{7} and numerical results of Belotserkovskii \cite{N8}   on spatial periodic domains  was also employed by Bondarenko {\it et al.} \cite{N11} for the comparison with     the experimental observations.

 Since the  electrolytic fluid  of the laboratory model  of Bondarenko {\it et al.} \cite{N11} is bounded by two lateral  walls of the plane horizontal cell,     Thess \cite{N18} considered  (\ref{new2}) on the duct domain $0\leq y \leq 2N\pi$   with the velocity field $\u$ satisfying slip boundary condition on the walls $y=0$ and $y=2N\pi$ for an integer $N>0$. He conducted numerical investigations into linear stability of  (\ref{new2}) in the duct and provided possible critical Reynolds number values comparable to those in experiment of \cite{N11}.  In contrast to the linear stability numerical results  of Thess \cite{N18},  Chen and Price \cite{N20} proved that (\ref{new2}) in the  duct bounded by slip boundary walls  $y=0$ and $y=2\pi$  (i.e. $N=1$), all
possible secondary flows transitional from the basic flow $u_0$ are self oscillations.
That is, the instabilities arising were analytically proved to be Hopf
bifurcations which were subsequently verified by simple numerical predictions. The secondary steady state flows observed by Bondarenko {\it et al.} \cite{N11}  arise only when $N>1$.  Chen  and Price \cite{N21}  proved the existence of  critical Reynolds number values resulted from the real linear spectral problem of (\ref{new2}) in the duct $0<y<2N\pi$. Then (\ref{new2}) is spectrally  truncated in  an infinite Fourier mode space containing  the basic flow and critical eigenfunctions. A circle of secondary steady-state flows supercritically  bifurcated from the basic flow in the truncated subspace is constructed analytically and is comparable with the secondary flows observed  by Bondarenko {\it et al.} \cite{N11}. Recently, the study of the Hartmann layer  friction effect  leads to the introduction of dissipative free surface Green function approach to  wave-structure interaction in  hydrodynamics (see Chen \cite{N22}).

Frenkel \cite{N23} considered  a quasi-normal mode approach in examining the
linear stability of periodic flows. This approach was further developed   by Zhang \cite{N24},  Zhang  and  Frenkel \cite{N25} to investigate linear stability problems.
Zhang \cite{N24} showed that intermediate-scale nonlinear instability of multidirectional
periodic flows is mathematically modelled by the Landau equation.

However, the existence of the secondary steady state solution to  wall-bounded fluid motion model  observed by Bondarenko {\it et al.} \cite{N11} in MHD laboratory experiment is still not demonstrated rigorously  in mathematical analysis.

The motivation for the present investigation is to prove  the existence of the secondary steady-state  flows  of (\ref{new2}) in the wall-bounded  domains for $N>1$.
As observed  by Chen and Price  \cite{N21}, the linearized equation of (\ref{new2}) under the wall slip  boundary condition admits a two-dimensional  eigenfunction space  at a single critical Reynolds number and there is  no flow invariant subspace containing a single eigenfunction, which is a crucial technical  condition necessary in   steady-state bifurcation theory (see, for example, Krasnoselskii \cite{N26}, Rabinowitz \cite{N27,N28}).
Inspired by the phase transition technique recently developed  by Ma and Wang \cite{N29_,N29} using the centre manifold theory, we find that topological structure for the exchange of  stability and  instability
of the basic flow around  a critical Reynolds number can be seen clearly in its centre manifold. Therefore, the existence of the steady-state bifurcation is proved.

 The present  theoretical predictions are in a good agreement  with the laboratory experimental observations of Bondarenko {\it et al.} \cite{N11} when $R=O(10^3)$. Further transition of the secondary flow to well developed turbulence is presented numerically for $R=O(10^4)$.

The steady-state bifurcation derived by the center manifold theory  lies on the analytical construction of the critical eigenfunctions, which is based on the spectral analysis of Chen and Price \cite{N20,N21} by using continued  fraction technique initiated from Mishalkin and Sinai \cite{6} and developed from Frenkel \cite{N23} and Zhang \cite{N24}.


\section{Real spectral problem}
The stream function $\psi$ solving (\ref{new2}) in  the wall-bounded  domain  is subject to the slip  boundary condition \cite{N18}
\bbe \label{bd1} \psi =\Delta \psi =0 \,\,\,\, \mbox{ at } y=0,\,\, y=2N\pi
\bee
and the periodic boundary condition in the $x$ direction
\bbe\psi(-\pi/k,y)=\psi(\pi/k,y),\,\,\,\, 0\leq y \leq 2N\pi.\label{bd2}
\bee
The Kolmogorov flow  is modified as
\bbe \psi_0 = \frac1{1+\lambda}\sin y.
\bee

By using the perturbation
\bbe \hat\psi=\psi-\psi_0,\bee
 equation (\ref{new2}) is written as, after omitting the superscript `hat',
 \bbe\label{n1a}
0=-\p_t \Delta \psi+L\psi+ \p_x \psi \,\p_y\Delta \psi-\p_y\psi\,\p_x\Delta \psi
\bee
with the linear operator
\bbe L\psi  = -\frac\lambda R \Delta \psi+\frac1 R \Delta^2\psi -\frac1{1+\lambda}\cos y (\Delta +1)\p_x\psi.
\bee

The linearisation of (\ref{n1a}) gives
\bbe\label{linear}
0=-\p_t \Delta \psi+L\psi
\bee
By taking $\psi= \e^{t\sigma/R}\psi'$ and omitting the superscript `prime', we have the spectral problem of (\ref{n1a})
\bbe\label{n1b}
0=-\frac\sigma   R\Delta\psi+L\psi 
\bee
for a real eigenvalue $\sigma $.
By the Fourier expansion, the eigenfunction of the spectral problem (\ref{n1b}) together with the  boundary conditions  (\ref{bd1}) and  (\ref{bd2})  can be expressed as
\bbe
\psi =\e^{\i k_xx}\sum_{n\in Z} \i^n \phi_n \sin(n+\frac j{2N})y. \label{nnn5}
\bee
for a wave number $k_x$ and $\i=\sqrt{-1}$. Here, for convenience, we use the explicit form of the factor $\i^n$ to ensure  $\phi_n$ to be real (see  Lemma \ref{L1}).

The existence of the real eigenvalue $\sigma$  and critical Reynolds number $R_c=R|_{\sigma =0}$ has been obtained.
\begin{lemma}\label{L1} (Chen and Price \cite[Lemma 2.1]{N21} ) Let $\lambda\ge 0$. Assume that wave number $k_x>0$ and   the integers $N\ge 2$ and $j=1,..., N-1$ are subject to the condition
\bbe \label{be}  k_x^2 + (\frac j{2N})^2<1  \,\,\,\mbox{ and }\,\,\, k_x^2+(1-\frac j{2N})^2>1.
\bee
Then for $\sigma   > -\lambda -k_x^2-(\frac j{2N})^2$,  there exists a unique value  $R$ so that the spectral problem (\ref{n1b}) and (\ref{nnn5}) has an eigenfunction solution $\psi$.   The  coefficients $\phi_n$ of the eigenfunction (\ref{nnn5}) is uniquely  determined as, up to a real constant factor,
\be \phi_0=1\,\,\,\,\mbox{ and } \,\,\,\,\phi_{\pm n} =\frac{\beta_0-1}{\beta_{\pm n}-1} \gamma_{\pm1}\cdots\gamma_{\pm n}, \,\,\, n\ge 1
\ee
for $\beta_{\pm n}= k_x^2+(\frac j{2N}\pm n)^2$ and
\bbe\label{nnn6}
\gamma_{\pm n} =\lim_{m\to \infty } \frac{\mp1}{\frac{2(1+\lambda)\beta_{\pm n}(\sigma  +\lambda + \beta_{\pm n})}{Rk_x(\beta_{\pm n}-1)}+ \frac1{\ddots + \frac1{\frac{2(1+\lambda)\beta_{\pm m}(\sigma  +\lambda + \beta_{\pm m})}{Rk_x(\beta_{\pm m}-1)}}}}.
\bee
\end{lemma}
The convergence of the continued fractions in (\ref{nnn6})  is due to Wall \cite[Theorem 30.1]{N30}
and Khinchin \cite[Theorem 10]{N31}.

The Hilbert space  $L_2=L_2([-\frac\pi {k_x},\frac\pi{k_x}]\times [0,2N\pi])$ is associated with the  boundary conditions  (\ref{bd1}) and  (\ref{bd2}) and the inner product
\be \langle \psi,\phi\rangle = \int^{\pi/k_x}_{-\pi/k_x}\int^{2N\pi}_0 \psi\overline{ \phi} dxdy.
\ee
Thus the dual pairing
$$\langle L\psi,\phi\rangle = \langle \psi,L^*\phi\rangle $$
defines  the conjugate operator
\bbe
 L^*\psi^*=-\frac\lambda R \Delta \psi^*+\frac1 R \Delta^2\psi^*+\frac1{1+\lambda} (\Delta +1)\cos y \p_x\psi^*
\bee
and the conjugate spectral problem
\bbe\label{n1c}
\frac\sigma   R\Delta \psi^*=L^*\psi^* 
\bee
associated with (\ref{bd1}) and (\ref{bd2}).
Hence we can write  the conjugate eigenfunction as
\bbe
\psi^* =\e^{\i k_xx}\sum_{n\in Z} \i^n \phi_n^* \sin(n+\frac j{2N})y \label{nnn7}
\bee
for coefficients $\phi_n^*$, to be shown  real in (\ref{realphi}).

\begin{remark}
When $j=N$, the eigenvalue $\sigma  $ of the spectral problem (\ref{n1b}) and (\ref{nnn5}) is a complex number, which  becomes  pure imaginary  at the corresponding  critical Reynolds number. The existence of Hopf bifurcation from the  Kolmogorov flow around the critical Reynolds number was  proved in \cite{N20}. In the present paper, we thus only consider case  $j=1,...,N-1$.
\end{remark}

\begin{lemma}\label{L2} Under the condition of Lemma \ref{L1}, let $\psi$ be  the eigenfunction given in Lemma \ref{L1} and $\psi^*$ be the associated the conjugate eigenfunction. Then we have
\bbe \langle - \Delta \psi,\psi^*\rangle  < 0 \mbox{ and } \langle \psi,\psi^*\rangle < 0.
\bee
\end{lemma}
\begin{proof}
By elementary manipulation, we have
\begin{align*}
\cos y (\Delta +1)\p_x\psi&= -k_x  \e^{\i k_xx } \sum_{n\in \Z} \i^{n+1}\cos y \phi_{n}(\beta_{n}-1)\sin(n+\frac j{2N})y
\\
&= \frac12 k_x  \e^{\i k_xx } \sum_{n=-\infty}^\infty \i^{n} [ (\beta_{n+1}-1)\phi_{n+1}-(\beta_{n-1}-1)\phi_{n-1}] \sin(n+\frac j{2N})y
\end{align*}
and
\begin{align*}
 (\Delta +1)\cos y\p_x\psi^*&
= \frac12 k_x  \e^{\i k_xx } \sum_{n=-\infty}^\infty \i^{n}  (\beta_{n}-1)(\phi_{n+1}^*-\phi_{n-1}^*) \sin(n+\frac j{2N})y.
\end{align*}
Therefore, the spectral problem (\ref{n1b}) and (\ref{nnn5}) and its conjugate problem reduce respectively to the algebraic equations
\begin{align}
\frac{2(1+\lambda)\beta_n(\sigma  +\lambda  +\beta_n)}{Rk_x} \phi_n=\phi_{n+1}  (\beta_{n+1}-1)- \phi_{n-1} (\beta_{n-1}-1),\,\,\, n\in Z,\label{nnn8}
\\
\frac{2(1+\lambda)\beta_n(\sigma  +\lambda + \beta_n) }{Rk_x}\phi_n^*=- (\beta_{n}-1)(\phi_{n+1}^* - \phi_{n-1}^*),\,\,\, n\in Z.\label{n1}
\end{align}
By (\ref{nnn8}) and (\ref{n1}), we have, up to a constant factor,
\bbe \label{realphi}\phi^*_n = (-1)^n(\beta_{n}-1)\phi_{n}
\bee
and hence
\bbe \langle  \Delta \psi,\psi^*\rangle  &=& \sum_{n \in \Z} (-1)^n\beta_n(\beta_n-1)\phi_n^2
< \sum_{n  \in \Z} \beta_n(\beta_n-1)\phi_n^2, \label{nnn9}
\\
 \langle   \psi,\psi^*\rangle  &=& \sum_{n \in \Z} (-1)^n(\beta_n-1)\phi_n^2
< \sum_{n  \in \Z} (\beta_n-1)\phi_n^2, \label{nnn91}
\bee
where we have used the condition $\beta_0-1<0$ and $\beta_{-1}-1 >0$ given in (\ref{be}), which ensures $\beta_n-1>0$ for $\n\neq 0$.

On the other hand, multiplying the $n$th equation of (\ref{nnn8}) by $(\beta_n-1)\phi_n$ and summing the resultant equations, we have
\bbe  0&=&\sum_{n \in \Z} \beta_n(\sigma  +\lambda  +\beta_n)(\beta_n-1)\phi_n^2\label{a00}
\\
&>& \sum_{n \in \Z} \beta_n(\sigma  +\lambda  +\beta_0)(\beta_n-1)\phi_n^2
\\
&>& \sum_{n \in \Z} \beta_0(\sigma  +\lambda  +\beta_0)(\beta_n-1)\phi_n^2.
\bee
This gives
\bbe\label{zero}
 \sum_{n  \in \Z} \beta_n(\beta_n-1)\phi_n^2<0 \mbox{ and } \sum_{n  \in \Z} (\beta_n-1)\phi_n^2<0.
 \bee
The combination of the previous equation with (\ref{nnn9}) and (\ref{nnn91})  yields the desired inequalities
\bbe \langle  \Delta \psi,\psi^*\rangle  &=& \sum_{n \in \Z} (-1)^n\beta_n(\beta_n-1)\phi_n^2
< 0, \label{nnn99}
\\
 \langle   \psi,\psi^*\rangle  &=& \sum_{n \in \Z} (-1)^n(\beta_n-1)\phi_n^2
< 0. \label{nnn991}
\bee
\end{proof}

\begin{lemma}
For the eigenvalue $\sigma  $ given in Lemma \ref{L1}, then we have
\bbe \frac{d\sigma  (R) }{dR} >0\,\,\,\mbox{ for }\,\,\, R>0\label{dR}
\bee
and
\bbe \label{R0}\lim_{R\to 0}\sigma  (R)=-\lambda -\beta_0 \mbox{ and }\,\,\, \lim_{R\to \infty} \sigma  (R)= \infty.
\bee\label{L3}
\end{lemma}

\begin{proof}
It is  implied  from  \cite{N20} that $\sigma $ is uniquely defined by $R$. For  the Hopf bifurcation problem with respect to complex eigenvalue problem, the corresponding positive derivative property has been proved in \cite{N20}. Equation (\ref{dR}) is now to be  obtained in a similar manner.

  Differentiate (\ref{nnn8}) with respect to $R$ to obtain
\begin{align*}
\lefteqn{\frac{2(1+\lambda)\beta_n\sigma  '}{Rk_x} \phi_n- \frac{2(1+\lambda)\beta_n(\sigma  +\lambda  +\beta_n)}{R^2k_x} \phi_n}\\
&=-\frac{2(1+\lambda)\beta_n(\sigma  +\lambda  +\beta_n)}{Rk_x} \phi_n' +\phi_{n+1}'  (\beta_{n+1}-1)- \phi_{n-1}' (\beta_{n-1}-1)
\end{align*}
for the superscript prime representing the partial derivative with respect to $R$.
Multiplying this equation by $(-1)^n(\beta_n-1)\phi_n$ and summing the resultant equations, we have
\begin{align*}
\lefteqn{\sum_{n  \in \Z} \frac{2(1+\lambda)(-1)^n\beta_n(\beta_n-1) \phi_n^2}{R k_x}\sigma  '- \sum_{n  \in \Z}\frac{2(1+\lambda)\beta_n(\sigma  +\lambda  +\beta_n)(-1)^n(\beta_n-1)}{R^2k_x} \phi_n^2 }\\
&=\sum_{n  \in \Z} \left(-\frac{2(1+\lambda)\beta_n(\sigma  +\lambda  +\beta_n)}{Rk_x} \phi_n'+\phi_{n+1}'  (\beta_{n+1}-1)- \phi_{n-1}' (\beta_{n-1}-1)\right)(-1)^n(\beta_n-1)\phi_n
\\
&=\sum_{n  \in \Z} \left(-\frac{2(1+\lambda)\beta_n(\sigma  +\lambda  +\beta_n)}{Rk_x} \phi_n+\phi_{n+1}  (\beta_{n+1}-1)- \phi_{n-1} (\beta_{n-1}-1)\right)(-1)^n(\beta_n-1)\phi_n',
\end{align*}
which equals zero due to (\ref{nnn8}). Thus we have
\begin{align}\label{aaa}
\sum_{n  \in \Z} (-1)^n\beta_n(\beta_n-1) \phi_n^2\sigma  '= \sum_{n  \in \Z}\frac{\beta_n(\sigma  +\lambda  +\beta_n)(-1)^n(\beta_n-1)}{R} \phi_n^2.
\end{align}
It follows from  (\ref{a00}) that the right-hand side of (\ref{aaa}) becomes
\begin{align*}
\frac1R\sum_{n  \in \Z}\beta_n(\sigma  +\lambda  +\beta_n)(-1)^n(\beta_n-1)\phi_n^2<\frac1R\sum_{n  \in \Z}\beta_n(\sigma  +\lambda  +\beta_n)(\beta_n-1)\phi_n^2=0.
\end{align*}
The combination of the previous equation with (\ref{nnn99}) and (\ref{aaa}) shows
\begin{align*}
\sigma  '= \frac1R\frac{\sum_{n  \in \Z}\beta_n(\sigma  +\lambda  +\beta_n)(-1)^n(\beta_n-1) \phi_n^2}{\langle \Delta \psi,\psi^*\rangle }>0,
\end{align*}
or the validity of (\ref{dR}).

For the proof of (\ref{R0}), we see that the spectral problem (\ref{n1b}) and (\ref{nnn5}) is equivalent to the continued fraction equation  \cite[Equation (2.8)]{N21}, which can be expressed as
\begin{align*}
\frac{2\beta_0(1+\lambda)(\sigma +\lambda  +\beta_0)}{Rk_x(1-\beta_0)}
=&\lim_{n\to \infty}\frac{1}{\frac{2\beta_1(1+\lambda)(\sigma  +\lambda +\beta_1)}{R k_x(\beta_1-1)}+\frac{1}{\ddots +\frac{1}{ \frac{2\beta_n(1+\lambda)(\sigma  +\lambda +\beta_n)}{Rk_x(\beta_n-1)}}}}
\\
&+\lim_{n\to \infty}\frac{1}{\frac{2\beta_1(1+\lambda)(\sigma  +\lambda +\beta_1)}{R k_x(\beta_1-1)}+\frac{1}{\ddots +\frac{1}{ \frac{2\beta_n(1+\lambda)(\sigma  +\lambda +\beta_n)}{Rk_x(\beta_n-1)}}}}
\end{align*}
or, by multiplying $R/(\sigma  +\lambda +\beta_0)$ to the previous equation,
\begin{align}
&\frac{2\beta_0(1+\lambda)}{k_x(1-\beta_0)}\label{xx}
\\
&=\lim_{n\to \infty}\frac{1}{\frac{2\beta_1(1+\lambda)(\sigma  +\lambda +\beta_1)(\sigma +\lambda  +\beta_0)}{R^2 k_x(\beta_1-1)}+\frac{1}{\ddots +\frac{1}{ \frac{2\beta_n(1+\lambda)(\sigma  +\lambda +\beta_n)(\sigma +\lambda  +\beta_0)^{(-1)^{n+1}}}{R^{1-(-1)^n}k_x(\beta_n-1)}}}}\nonumber
\\
&+\lim_{n\to -\infty}\frac{1}{\frac{2\beta_{-1}(1+\lambda)(\sigma  +\lambda +\beta_{-1})(\sigma +\lambda  +\beta_0)}{R^2 k_x(\beta_{-1}-1)}+\frac{1}{\ddots +\frac{1}{ \frac{2\beta_n(1+\lambda)(\sigma  +\lambda +\beta_n)(\sigma +\lambda  +\beta_0)^{(-1)^{n+1}}}{R^{1-(-1)^n}k_x(\beta_n-1)}}}}.\nonumber
\end{align}
The left-hand side of (\ref{xx}) is a constant with respect to $\sigma $ and $R$. Since $\sigma  (R)$ is a continuous function of $R>0$, The action  $R\to 0$ in (\ref{xx}) shows that
$$\lim_{R\to 0} \frac{\sigma  (R)+\lambda+\beta_0}{R^2} =0.$$

On the other hand, if $\sigma  (R)$ is uniformly bounded for $R>0$, then  $R\to \infty$ in (\ref{xx}) shows that the right-hand side approaches to infinity, while  the left-hand side of (\ref{xx}) remains constant. Thus the boundedness assumption of $\sigma  (R)$ is not true. This gives the validity of  (\ref{R0}).
\end{proof}
\section{Existence of secondary steady-state flows}

Upon the observation of the spectral problem in the previous section, we have   the function $\sigma (R)$ for $R>0$ or its inverse $R(\sigma )$  for  $\sigma   >-\lambda-\beta_0$. This gives the existence of the critical Reynolds number $R_c= R(\sigma =0)$ , which also depends on $k_x$ and $j=1,...,N-1$. Thus it is expected to have the existence of steady-state solutions bifurcating from $\psi_0$ as $R$ varies across $R_c$. However the eigenfunction space spanned  by the two orthogonal real eigenfunctions
\bbe
\psi_1&=&\R \psi =\sum_{n\in Z} \phi_n \cos ( k_x x+\frac{n\pi}2)\sin(n+\frac j{2N})y, \label{nnn10}
\\
\psi_2&=&\I \psi =\sum_{n\in Z} \phi_n \sin ( k_x x+\frac{n\pi}2)\sin(n+\frac j{2N})y. \label{nnn101}
\bee
Actually, for any  flow invariant space of (\ref{new2})-(\ref{bd2}) containing  one of the  eigenfunctions above, it must contains the another one as well. Thus we cannot use steady-state bifurcation theorems, as they are not applicable to  even-dimensional eigenfunction space problem. Recently, Ma and Wang \cite{N29_,N29} use central manifold theorem to reduce a partial differential equation to  an ordinary differential equation with respect to the center manifold to find topological structure transition around the bifurcation point.  In the present paper, we will follow this argument to show the bifurcation into a circle of steady-state solutions as the Reynolds varies across the critical value $R_c$.

For  the function spaces  $L_2$ space
under the norm
\begin{align*} \|\psi\|_{L_2} =\left(\int^{\pi/k_x}_{-\pi/k_x} \int^{2N\pi}_0 |\psi|^2dxdy\right)^{1/2},
\end{align*}
we consider the solution in  the Sobolev space
\begin{align*}
H^4 = \left\{ \psi \in L_2; \,\,  \|\psi\|_{H^4}=\|\Delta^{2}\psi\|_{L_2}<\infty, \,\, \psi \mbox{ satisfies the conditions (\ref{bd1})-(\ref{bd2})}\right\}.
\end{align*}

By Lemma \ref{L3}, the spectral solution  $(\sigma , \psi)$ of the spectral problem \eqref{n1b} and \eqref{nnn5} is uniquely defined by the Reynolds number $R$  for given  $k_x$ and $j$.
Thus the critical Reynolds number $R_c=R_{c,k_x,j}$  is uniquely defined. However, we cannot prove that
\bbe \label{R}R_{c,k_x,j}\ne R_{c,k_x',j'} \,\,\mbox{ whenever } (k_x,j)\ne (k_x',j'),
\bee
although (\ref{R}) always true by numerical simulation.

To  use the centre manifold theorem for $R$ close to the critical value $R_c$,
we define the nonlinear operator
\begin{align*}  N(f,g ) =\p_x f \,\p_y\Delta g -\p_yf\,\p_x\Delta g
\end{align*}
and assume  $\psi$  the eigenfunction (\ref{nnn5}) in the remaining of this section.  For convenience of notation in the present section, we  let  $\varphi$ be  the unknown stream function of the Navier-Stokes equation (\ref{n1a}) under the boundary condition (\ref{bd1})-(\ref{bd2}). Thus $\varphi$ solves the  dynamical euqation
\bbe \p_t \Delta \varphi = L\varphi+ N(\varphi,\varphi),\,\,\, \varphi \in H^4.\label{equ}
\bee
Recall the  conjugate eigenfunction $\psi^*$ in (\ref{nnn7}).  We  use the real conjugate eigenfunctions
\bbe
\psi_1^*&=&\R \psi^* =\sum_{n\in Z} \phi_n^* \cos ( k_x x+\frac{n\pi}2)\sin(n+\frac j{2N})y, \label{nnn10}
\\
\psi_2^*&=&\I \psi^* =\sum_{n\in Z} \phi_n^* \sin ( k_x x+\frac{n\pi}2)\sin(n+\frac j{2N})y \label{nnn101}
\bee
with respect respectively to $\psi_1$ and $\psi_2$.

Define the central space  and the stable space
$$E_c = \left\{ s_1 \psi_1+s_2\psi_2|\,\,  (s_1,s_2)\in R^2\right\},$$
$$E_s =\left\{ \phi\in H^4 | \,\,\, \langle \phi,\psi_1^*\rangle =\langle \phi,\psi^*_2\rangle =0\right\}.$$
Employ  Lemma \ref{L2} to define the projection operator
$${\mathcal{P}}\phi = \phi -\frac{\langle \phi,\psi_1^*\rangle }{\langle \psi_1,\psi^*_1\rangle }\psi_1-\frac{\langle \phi,\psi_2^*\rangle }{\langle \psi_2,\psi^*_2\rangle }\psi_2,
$$
which  ensures  $E_s={\mathcal{P}}E_s$ and ${\mathcal{P}}$ maps  $L_2$ onto ${\mathcal{P}}L_2$.
 It follows   Lemma \ref{L3} that $\sigma  (R)$ is strictly monotone function of $R$ as $R$ increase across the critical Reynolds number $R_c$.
 By the assumption (\ref{R}),  $R_c=R_{c,k_x,j}$ is a unique critical Reynolds number for given $k_x$ and $j$.   Thus   by the Sobolev imbedding theorem and the  Fredholm alternative, the linear operator,  with  $\sigma $ in a vicinity of $\sigma =0$,
\begin{align*} L : E_s \mapsto {\mathcal{P}}L_2
\end{align*}
is a bijection and has the bounded inverse
\begin{align*} L^{-1} : {\mathcal{P}}L_2 \mapsto E_s.
\end{align*}
It is  also readily seen that the nonlinear operator $N: H^4\times H^4\mapsto L_2$ is  compact.

We  are in the position to state the main result of the present paper:
\begin{theorem}\label{th}
Let $N>1$ and let  $k_x$ and $j=1,...,N-1$ satisfy   the condition
\bbe \label{be1}\beta_0-1<0,\,\,\, \mbox{ }\,\,\, \beta_{-1}-1>0
\bee
by recalling $\beta_n=k_x^2 +(n+\frac j{2N})^2$. In addition to the condition (\ref{R}), assume that
\begin{align*}
& \langle  N( \bar\psi, L^{-1}N(\psi,\psi)),\psi^*\rangle +\langle  N(  \psi,  L^{-1}N(\psi,\bar\psi)+L^{-1}N(\bar\psi,\psi)  ),\psi^*\rangle
 \\
  +& \langle  N(  L^{-1}N(\psi,\psi),\bar\psi),\psi^*\rangle +\langle  N(  L^{-1}N(\psi,\bar\psi)+L^{-1}N(\bar\psi,\psi) ,\psi ),\psi^*\rangle \neq 0.
\end{align*}\label{T1}
Then (\ref{new2})  in $H^4$ admits a circle of  steady-state solutions  branching off  $\psi_0$ as $R$ varies across the critical Reynolds number  $R_c= R_{c,k_x,j}$.
\end{theorem}

\begin{proof}

By Lemma \ref{L2}, we may assume the normalization
\bbe \label{D}\langle \Delta\psi_i,\psi^*_j\rangle =\delta_{i,j},\,\,\,i,\,j=1,2
\bee
for $\delta_{i,j}$ the Kronecker delta function.
The  unknown stream function $\varphi$ is written as the orthogonal decomposition form
\begin{align*} \varphi= s_1 \psi_1+s_2\psi_2+\phi\,\,\,\mbox{ for }\,\,\, \phi \in E_s.
\end{align*}
This yields
\begin{align*} L\varphi = \mu  s_1 \Delta \psi_1 + \mu s_2 \Delta\psi_2 +L\phi\,\,\,\mbox{ for }\,\, \mu=\frac \sigma R.
\end{align*}
Hence, with the  use of (\ref{D}), we can rewrite (\ref{equ}) with $(\varphi, \mu)$ in a vicinity of $(0,0)$ into the following dynamical system
\bbe\label{mm}
\left\{
\begin{aligned}\displaystyle\frac{ds_1}{dt}&=\mu s_1+ \langle N(s_1 \psi_1+s_2\psi_2+\phi, s_1 \psi_1+s_2\psi_2+\phi),\psi_1^*\rangle, \vspace{2mm}
\\
\displaystyle\frac{ds_2}{dt}&= \mu s_2+ \langle N(s_1 \psi_1+s_2\psi_2+\phi, s_1 \psi_1+s_2\psi_2+\phi),\psi_2^*\rangle, \vspace{2mm}
\\
\frac{d \Delta\phi}{dt}&= L\phi+ {\mathcal{P}}N(s_1 \psi_1+s_2\psi_2+\phi, s_1 \psi_1+s_2\psi_2+\phi),
\\
\frac{d\mu}{dt}&=0.
\end{aligned}
\right.
\bee
By the centre manifold theorem, there exists a center manifold function presented in the Taylor expansion
$$M= s_1^2\chi_ {1,1}+s_1s_2\chi_ {1,2}+s_1s_2\chi_ {2,1}+ s_2^2 \chi_ {2,2}+  s_1\mu\chi_ 1 +s_2\mu\chi_ 2+ \mu^2\chi_ 0  +o(|(s_1,s_2,\mu)|^2)$$
for $\chi_ {i,j}, \chi_ i \in E_s$.
The function $M$ is tangential to the centre space of the system (\ref{mm})
and satisfies $\phi=M$ in a small neighbourhood of $(s_1,s_2,\mu)=(0,0,0)$. That is, equation (\ref{mm}) becomes
\bbe \frac{ds_1}{dt}&=& \mu s_1+ \sum_{i=1}^2 s_i\langle N( \psi_i,M)\!+\!N(M,\psi_i),\psi^*_1\rangle \!+\!\langle N(M,M),\psi_1^*\rangle, \label{m1}
\\
\frac{ds_1}{dt}&=& \mu s_2+  \sum_{i=1}^2 s_i\langle N( \psi_i,M)\!+\!N(M,\psi_i),\psi^*_2\rangle \!+\!\langle N(M,M),\psi_2^*\rangle, \label{m2}
\\
 \frac{d \Delta M}{dt} &=& L M+ {\mathcal{P}}N(s_1 \psi_1+s_2\psi_2+M, s_1 \psi_1+s_2\psi_2+M),\label{m3}
\bee
where we have used the property
\bbe\langle N(\psi_i,\psi_j),\psi_l^*\rangle =0 \,\,\mbox{ or }\,\,\, {\mathcal{P}}N(\psi_i,\psi_j)=N(\psi_i,\psi_j),\label{PN}
\bee
due to  the definition of the eigenfunctions $\psi_i$ and the conjugate eigenfunctions $\psi_i^*$.

To find the dynamical behaviour of the system around $(s_1,s_2,\mu)=(0,0,0)$, it is crucial to determine the  functions $\chi_ {i,j}$ and $\chi_ i$ in the  principal part of $M$.

Indeed, by (\ref{m3}) and (\ref{PN}), we have
\begin{align}\label{nm}
\frac{d\Delta  M}{dt}&=
 L[ \sum_{i,j=1}^2s_is_j\chi_ {i,j}+ \chi_ 1 s_1\mu +\chi_ 2 s_2\mu + \chi_ 0 \mu^2+o(|(s_1,s_2,\mu)|^2]\nonumber
 \\
 &+
\sum_{i,j=1}^2N(\psi_i,\psi_j)s_is_j  +\sum_{i=1}^2s_i( {\mathcal{P}}N(\psi_i,M)+{\mathcal{P}}N(M,\psi_i)) .
\end{align}

On the other hand, by (\ref{m1}), (\ref{m2}) and (\ref{PN}), we have
\begin{align}
\frac{d \Delta  M}{dt}&=\frac{\partial\Delta M}{\p s_1} \frac{ d s_1}{dt} +\frac{\partial \Delta M}{\p s_2} \frac{ d s_2}{dt}\label{nm1}
\\
&=\Delta(2s_1\chi_ {1,1}\!+\!s_2\chi_ {1,2}\!+\!s_2\chi_ {2,1} \!+\!\chi_ 1\mu \!+\! o(|(s_1,s_2,\mu)|^3)\nonumber
\\
&\times (\mu s_1 \!+\! \sum_{i=1}^2 s_i\langle N( \psi_i,M)\!+\!N(M,\psi_i),\psi^*_1\rangle \!+\!\langle N(M,M),\psi_1^*\rangle )\nonumber
\\
&\!+\!\Delta (2s_2\chi_ {2,2}\!+\!s_1\chi_ {1,2}\!+\!s_1\chi_ {2,1} \!+\!\chi_ 2\mu\!+\! o(|(s_1,s_2,\mu)|^3 )\nonumber
\\
&\times (\mu s_2 \!+\!  \sum_{i=1}^2 s_i\langle N( \psi_i,M)\!+\!N(M,\psi_i),\psi^*_2\rangle \!+\!\langle N(M,M),\psi_2^*\rangle ).\nonumber
\end{align}
Note  that (\ref{nm}) = (\ref{nm1}). We find
$$ L\chi_ {i,j}=-N(\psi_i,\psi_j), \,\,\,\chi_ i=0, \,\,\,\chi_ 0=0$$
and thus
$$M= -\sum_{i,j=1}^2 s_is_j L^{-1} N(\psi_i,\psi_j) +o(|(s_1,s_2,\mu)|^2).$$
Higher order terms of $M$ can also be obtained from the balance of the equation (\ref{nm})=(\ref{nm1}) and the centre manifold function can be further obtained as
 \bbe M = -\sum_{i,j=1}^2 s_is_j L^{-1} N(\psi_i,\psi_j) +O(|(s_1,s_2)|^3)+ O(|\mu|\,|(s_1,s_2)|^2).\label{hh}\bee

With the use of this expression, we can reduce (\ref{m1}) and (\ref{m2}) to an equation system, which is only dependent of $(s_1,s_2,\mu)$ in a small neighbourhood of $(s_1,s_2,\mu)=(0,0,0)$. It remains to simplify the nonlinear terms of (\ref{m1}) and (\ref{m2}) by using (\ref{hh}).

To do so, we use  the complex formulation
 $$\x=s_1+\i s_2,\,\,\, \psi=\psi_1+\i\psi_2,\,\,\, \psi^*=\psi_1^*+\i \psi^*_2$$
to obtain
\begin{align*}
N(\psi_1,\psi_1) &= \frac14 N(\psi+\bar\psi,\psi+\bar\psi) =  \frac14[N(\psi,\psi)+N(\bar\psi,\psi)+N(\psi,\bar\psi)+N(\bar\psi,\bar\psi)],
\\
N(\psi_1,\psi_2) &=-\i \frac14 N(\psi+\bar\psi,\psi-\bar\psi) = -\i \frac14[N(\psi,\psi)+N(\bar\psi,\psi)-N(\psi,\bar\psi)-N(\bar\psi,\bar\psi)],
\\
N(\psi_2,\psi_1) &=-\i \frac14 N(\psi-\bar\psi,\psi+\bar\psi) = -\i \frac14[N(\psi,\psi)-N(\bar\psi,\psi)+N(\psi,\bar\psi)-N(\bar\psi,\bar\psi)],
\\
N(\psi_2,\psi_2) &= -\frac14 N(\psi-\bar\psi,\psi-\bar\psi) =  \frac14[N(\psi,\psi)-N(\bar\psi,\psi)-N(\psi,\bar\psi)+N(\bar\psi,\bar\psi)]
\end{align*}
and hence
\begin{align*}
\lefteqn{4[ s_1^2N(\psi_1,\psi_1)+s_1s_2N(\psi_1,\psi_2)+s_1s_2N(\psi_2,\psi_1)+s_2^2N(\psi_2,\psi_2)]}
\\
&=s_1^2[N(\psi,\psi)+N(\bar\psi,\psi)+N(\psi,\bar\psi)+N(\bar\psi,\bar\psi)]-\i s_1s_2[N(\psi,\psi)+N(\bar\psi,\psi)-N(\psi,\bar\psi)-N(\bar\psi,\bar\psi)]
\\
&-\i s_1s_2[N(\psi,\psi)-N(\bar\psi,\psi)+N(\psi,\bar\psi)-N(\bar\psi,\bar\psi)]-s_2^2[N(\psi,\psi)-N(\bar\psi,\psi)-N(\psi,\bar\psi)+N(\bar\psi,\bar\psi)]
\\
&=(s_1^2-2\i s_1s_2-s_2^2)N(\psi,\psi)+(s_1^2+s_2^2)N(\bar \psi,\psi)+(s_1^2+s_2^2)N(\psi,\bar\psi)+(s_1^2+2\i s_1s_2-s_2^2)N(\bar\psi,\bar\psi)
\\
&=\bar\x^2N(\psi,\psi)+|\x|^2 N(\bar \psi,\psi)+|\x|^2N(\psi,\bar\psi)+\x^2N(\bar\psi,\bar\psi).
\end{align*}

Since   $L$ is unidirectional operator applying along  in the $y$-direction and the  eigenfunction  is in the form
$$\psi=\e^{\i k_xx} \sum\phi_n  \i^n \sin(n+\frac j{2N})y,$$
there exist  functions $f_i$ for $i=1,...,4$  independent of $x$ such that
\begin{align}\sum_{i,j=1}^2s_is_j\chi_ {i,j} &= -\frac14 L^{-1}[\bar\x^2N(\psi,\psi)+|\x|^2 N(\bar \psi,\psi)+|\x|^2N(\psi,\bar\psi)+\x^2N(\bar\psi,\bar\psi)]\nonumber
\\
&= \bar \x^2 \e^{2\i k_xx} f_1(y) + |\x|^2 f_2(y) + |\x|^2 f_3(y) + \x^2 \e^{-2\i k_xx} f_4(y).\label{v}
\end{align}
This together with the elementary fact
$$\int^{\pi/k_x}_{-\pi/k_x} e^{ \i m k_xx}dx =0,\,\,\, \mbox{ whenever }\,\,\, m\neq 0$$
implies
\begin{align*}
\sum_{i,j=1}^2& s_is_j\langle s_1N( \psi_1,\chi_ {i,j})+s_2N( \psi_2,\chi_ {i,j}),\bar\psi^*\rangle
\\
=&\frac12\sum_{i,j=1}^2 s_is_j\langle s_1N( \psi+\bar\psi,\chi_ {i,j})-\i s_2N( \psi-\bar\psi,\chi_ {i,j}),\bar\psi^*\rangle
\\
=&\frac12\sum_{i,j=1}^2 s_is_j\langle \bar \x N( \psi,\chi_ {i,j})+\overline{\bar \x N( \psi,\chi_ {i,j})},\bar\psi^*\rangle
\\
=& \frac12\bar \x\langle  N( \psi,\bar \x^2 \e^{2\i k_xx} f_1 + |\x|^2 f_2 + |\x|^2 f_3 + \x^2 \e^{-2\i k_xx} f_4),\bar\psi^*\rangle
\\
&+  \frac12\x\langle  N( \bar \psi,\bar \x^2 \e^{-2\i k_xx} \bar f_1 + |\x|^2 \bar f_2 + |\x|^2 \bar f_3 + \x^2 \e^{2\i k_xx} \bar f_4),\bar\psi^*\rangle
\\
=& \frac12  \bar \x\langle  N( \bar\psi, \x^2 \e^{2\i k_xx} \bar f_4),\bar\psi^*\rangle +\x\langle  N( \bar \psi, |\x|^2 \bar f_2 + |\x|^2 \bar f_3 ),\bar\psi^*\rangle
\\
=& \frac12  \x|\x|^2[ \langle  N( \psi, \e^{-2\i k_xx}  f_4),\bar\psi^*\rangle +\langle  N( \bar \psi,  \bar f_2 +  \bar f_3 ),\bar\psi^*\rangle ].
\end{align*}
That is, by (\ref{v}),
\begin{align}
\sum_{i,j=1}^2& s_is_j(s_1N( \psi_1,\chi_ {i,j})+s_2N( \psi_2,\chi_ {i,j}),\bar\psi^*)= -a \x|\x|^2 \label{a}
\end{align}
with
\begin{align*}
&a=   \frac18 ( \langle  N( \psi, L^{-1}N(\bar\psi,\bar\psi)),\bar\psi^*\rangle +\langle  N( \bar \psi,  L^{-1}N(\psi,\bar\psi)+L^{-1}N(\bar\psi,\psi)  ),\bar\psi^*\rangle).
\end{align*}
Similarly, we have
\begin{align}
\sum_{i,j=1}^2& s_is_j\langle s_1N(\chi_ {i,j},\psi_1)+s_2N( \chi_ {i,j},\psi_2),\bar\psi^*\rangle = -b \x|\x|^2\label{b}
\end{align}
with
\begin{align*}
&b=    \frac18(\langle  N(  L^{-1}N(\bar\psi,\bar\psi),\psi),\bar\psi^*\rangle +\langle  N(   L^{-1}N(\psi,\bar\psi)+L^{-1}N(\bar\psi,\psi) ,\bar\psi ),\bar\psi^*\rangle).
\end{align*}

With the use of (\ref{a}) and (\ref{b}) for  reduction of the nonlinear term of the $s_1$ and $s_2$ equations, we can now multiply (\ref{m2}) by the imaginary unit $\i$ and add the resultant equation to (\ref{m1}) to obtain
\begin{align*} \frac{d\x}{dt}=& \mu \x+ \sum_{i=1}^2 s_i\langle N( \psi_i,M)\!+\!N(M,\psi_i),\psi^*_1\rangle \!+\!\langle N(M,M),\bar\psi^*\rangle \label{m1m}
\\
=& \mu \x+ \sum_{i=1}^2 s_is_j\langle s_1(N( \psi_1,\chi_ {i,j})+s_2(N( \psi_2,\chi_ {i,j})+s_1N(\chi_ {i,j},\psi_1),\bar\psi^*\rangle
\\&+\sum_{i=1}^2 s_is_j\langle s_2N(\chi_ {i,j},\psi_2),\bar\psi^*\rangle+O(|\x|^4)+ O(|\mu| |\x|^3)
\\
=&\mu \x-(a+b)\x|\x|^2+O(|\x|^4)+ O(|\mu| |\x|^3).
\end{align*}

We thus have the supercritical bifurcation into a circle of solutions
$$\frac{\mu}{a+b}=|\x|^2+O(|\x|^3)  \,\,\mbox{ whenever }\,\,\, a+b >0$$
the subcritical bifurcation into a circle of solutions
$$\frac{\mu}{a+b}=|\x|^2+O(|\x|^3)  \,\,\mbox{ whenever }\,\,\, a+b < 0.$$
\end{proof}

This confirms the supercritical bifurcation into a circle of solutions for a simple spectral truncation of (\ref{equ}) (see Chen and Price \cite{N21}) as $a+b$ is always positive  by numerical computations.

\section{Numerical results}
As shown in Theorem  \ref{T1}, there exists a circle of steady-state solutions branching off the basic flow $\psi_0$ as $R$ varies across the critical Reynolds number $R_c=R_{c,k_x,j}$ satisfying (\ref{be}).
 The multiple solution bifurcation is  from the symmetry of the Navier-Stokes equation with respect to $x$. Therefore for any $\theta \in [0,2\pi)$,  a steady-state solution is bifurcating from $\psi_0$ in the direction of the eigenfunction
$$\psi_\theta= \sum_{n\in Z} \phi_n \i^n\cos(k_xx+\frac{n\pi}2+\theta)\sin (n+\frac{j}{2N})y,$$
which is a linear combination of the eigenfunctions $\psi_1$ and $\psi_2$ defined by the following modes
\bbe \label{mode}\cos(k_xx+\frac{n\pi}2)\sin (n+\frac{j}{2N})y\,\, \mbox{ and } \,\,\sin(k_xx+\frac{n\pi}2)\sin (n+\frac{j}{2N})y,\,\,\, n\in Z.
\bee
The  spectral truncation scheme \cite{N21} involving   the eigenfunction modes (\ref{mode}) and    forcing mode or the basic flow mode   $\sin y$ gives the first order approximation of the bifurcating solutions, since in the spirit of the bifurcation technique of Rabinowitz \cite{N27}  the  bifurcating solution of (\ref{new2})-(\ref{bd2})  is in the form
$$\psi=\psi_0+\varepsilon \psi_\theta + \varepsilon^2 \delta \psi,\,\,\,\, R= {R_c} + \varepsilon^2 \delta R
$$
for a function $\delta \psi$, a number $\delta R$ and a small parameter $\varepsilon$.
Integrating by parts, we see that  the solution
\begin{align*} \langle N(\psi,\psi),\Delta \psi\rangle = \langle N(\psi,\psi), \psi\rangle=0.
\end{align*}
Taking the inner product of (\ref{new2}) with $\Delta \psi +\psi$, we have
\bbe\label{psi} \langle \p_t \Delta \psi, \Delta\psi +\psi\rangle = \langle \frac{1}R(\Delta^2 +\lambda \Delta)\psi, \Delta \psi+\psi\rangle.
\bee
By Fourier expansion, the solution may expressed as
\begin{align*} \psi= \sum_{n,m,l} a_{n,m,l} \e^{\i mk_x x} \sin (n+\frac{lj}{2N})y,\,\,\, \bar a_{n,m,l}=a_{n,-m,l}
\end{align*}
for the summation of $n,m\in Z$ and  $l$ in a suitable $l$ set.
Thus (\ref{psi}) can be rewritten as
\bbe\label{psi1} \lefteqn{\sum_{n,m,l} \frac{d |a_{n,m,l}|^2}{dt} \beta_{n,m,l}(\beta_{n,m,l}-1)}\nonumber
\\ &&= \frac2R\sum_{n,m,l}  |a_{n,m,l}|^2 (\beta_{n,m,l}+\lambda)\beta_{n,m,l}(\beta_{n,m,l}-1)
 \bee
  for $$\beta_{n,m,l}= (mk_x)^2+(n+\frac{lj}{2N})^2.$$
 This shows that the solution is essentially dominated by a couple of the items $a_{n,m,l} \e^{\i mk_x x} \sin (n+\frac{lj}{2N})y$ so that $\beta_{n,m,l}<1$.  In particular,  for a steady-state solution $\psi$, the right-hand side of (\ref{psi1})  equals zero. This gives
 \bbe\label{Bij}\lefteqn{\sum_{n,m,l, \beta_{n,m,l}<1}  |a_{n,m,l}|^2 (\beta_{n,m,l}+\lambda)\beta_{n,m,l}(1-\beta_{n,m,l})\nonumber}
 \\
 &&=\sum_{n,m,l,\beta_{n,m,l}>1}  |a_{n,m,l}|^2 (\beta_{n,m,l}+\lambda)\beta_{n,m,l}(\beta_{n,m,l}-1).
 \bee
 This is the nonlinear extension of  the identity
 \bbe
  (\sigma  +\lambda+\beta_0)\beta_0(1-\beta_0) |\phi_0|^2 = \sum_{n\neq 0} (\sigma  +\lambda+\beta_n)\beta_n(\beta_n-1)|\phi_n|^2
  \bee
  for the linear eigenfunction
  $$ \sum_{n\in Z} \phi_n\i^n  \e^{\i k_x x} \sin (n+\frac j{2N})y.$$
 Thus $a_{n,m,l}$ tends to zero rapidly as $|n|$ and $|m|$ increase.
  With this observation, we use a spectral truncation scheme involving Fourier expansion modes with the selection of $-2\leq m\leq 2$, $n \in Z$ and suitable $l$'s.

  The numerical computation is related to the laboratory measurements given by Bondarenko {\it et al.} \cite{N11} with respect to the motion of  a thin layer of  the electrolytic fluid in a wall-bounded domain  in an electromagnetic field.
  In the laboratory experiment, critical  Reynolds number is around $2000$, $\lambda=20$  and the wave number $k_x =0.68\pm 0.05$. The   steady-state bifurcating solution  flow pattern  slightly over the critical Reynolds number $2000$ is displayed in \cite[Figures 4]{N11} (see also Figure \ref{f1}).
  \begin{figure}[h]
\centering \vspace{0mm}
\includegraphics[scale=0.35]{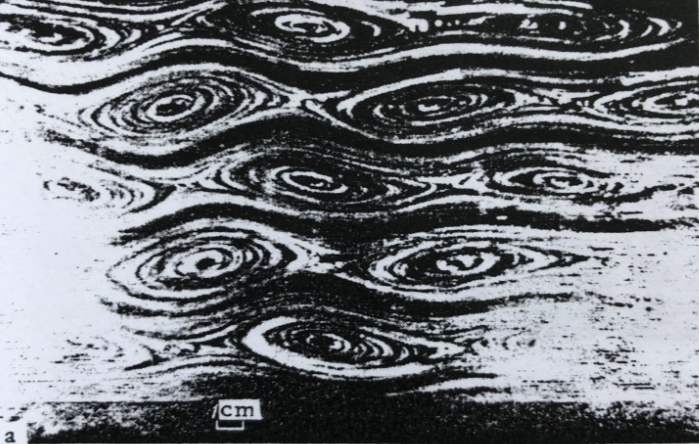}
\includegraphics[scale=0.52]{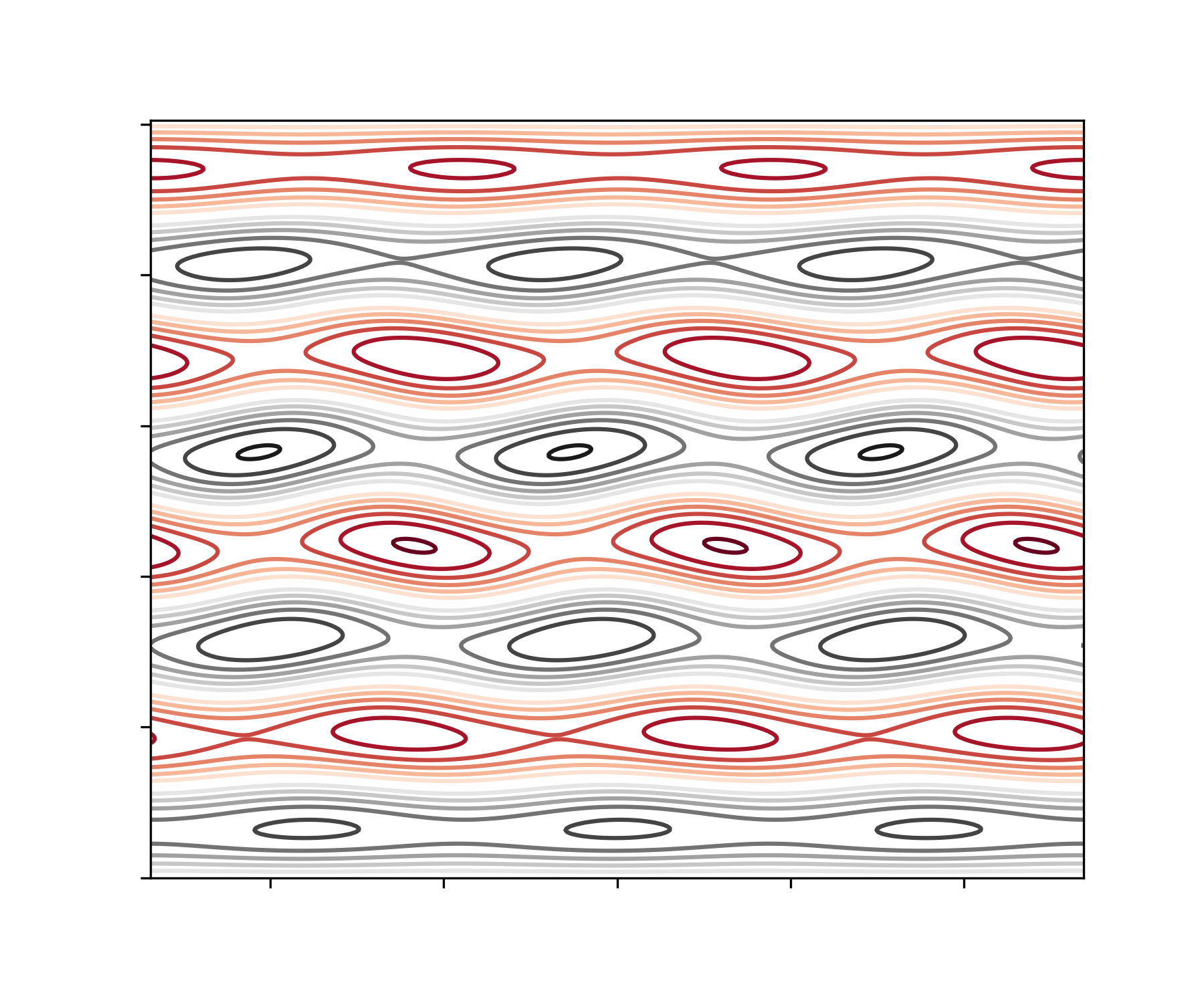}

\caption{Top: Experimental secondary steady-state flow, with the Reynolds number slightly over the critical experimental  Reynolds number 2000,  given by Bandarenko \cite[Figure 4]{N11} for $\lambda =20$ $k_x=0.68\pm 0.05$. Bottom: The secondary flow derived in Theorem \ref{th} for $\lambda=20$,  $k_x=0.7$, $N=4$,  $j=1$, $-3\pi/k_x\leq x \leq 3\pi/k_x$, $0 \leq y\leq 2N\pi$ and
$R=1810$, which is over the analytic critical Reynolds number $R_c=1760$.}\label{f1}
\end{figure}
\begin{figure}[h]
\centering \vspace{0mm}
\includegraphics[scale=0.45]{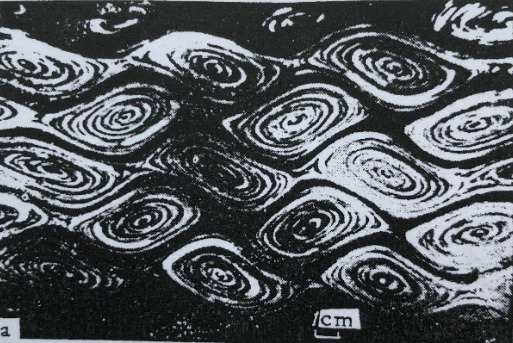}
\includegraphics[scale=0.52]{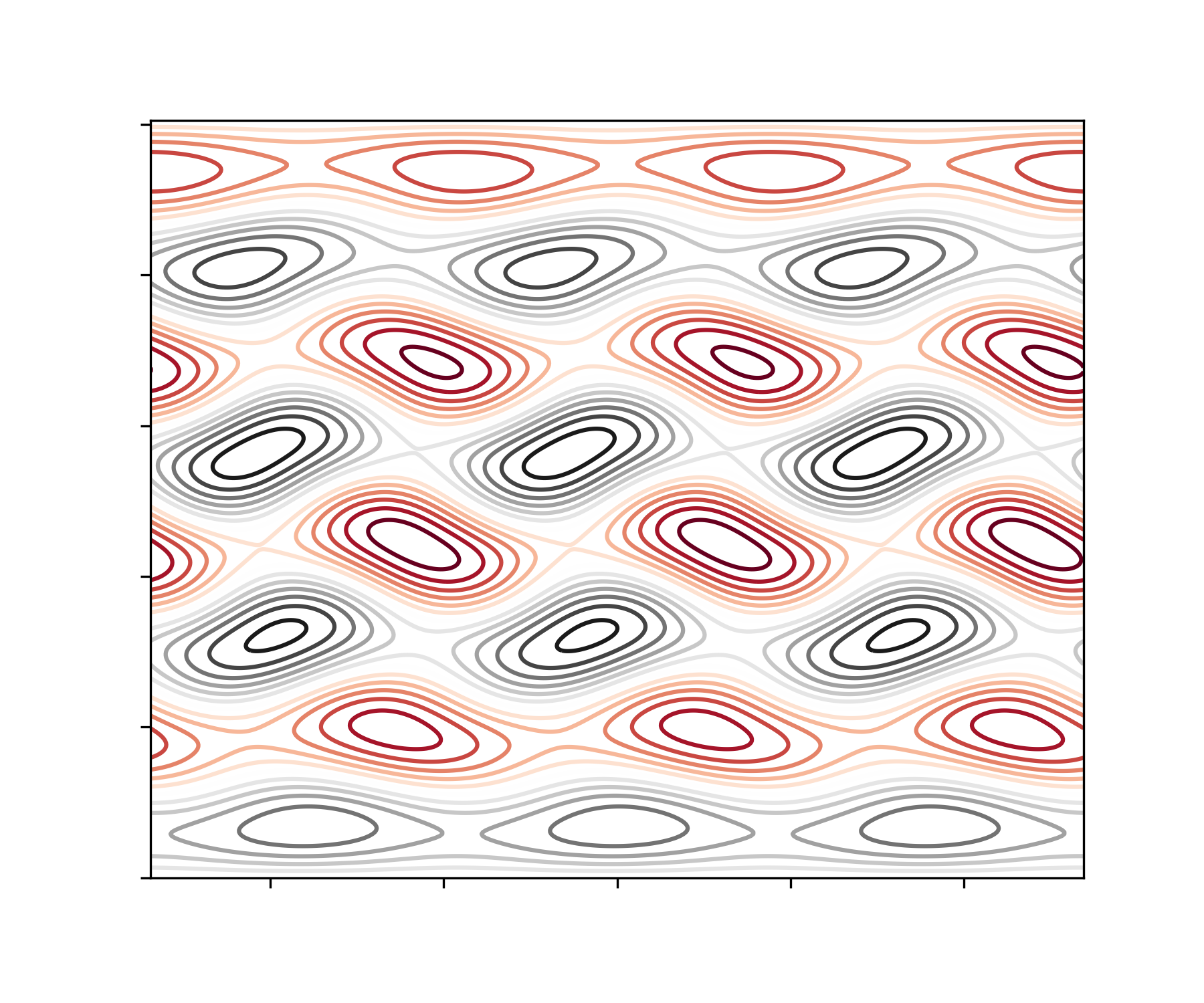}
\caption{ Top:  Experimental  flow pattern \cite[Figure 5]{N11} at $R=1.25\times 2000$, well above the experimental critical value $2000$. Bottom: The secondary flow derived in Theorem \ref{th} for $\lambda=20$,  $k_x=0.7$, $N=4$,  $j=1$, $-3\pi/k_x\leq x \leq 3\pi/k_x$, $0 \leq y\leq 2N\pi$ and
$R=2300$, which is well above  the analytic critical Reynolds number $R_c=1760$.}\label{f2}
\end{figure}
It has been given in \cite{N21} that it is suitable to take $N=4$ and $j=1$ to get numerical flow comparable with experimental secondary flow in Figures \ref{f1} and \ref{f2}. When $N=4$ and $j=1$, the first  critical  Reynolds number for the (\ref{new2}) and (\ref{bd1}) in the wall-bounded domain  is about 1768 (see \cite{N21}), which is reached at the wave number $k_x=0.63$. Now  we take $k_x=0.7$  due to $k_x=0.68\pm0.05$ in \cite{N11}. The secondary steady-state flow is dependent of  the phase number $\theta \in [0,2\pi)$ and is generated by the eigenfunction $\psi_\theta$ and the basic flow $\sin y$. Therefore by the phase transition
$x+\theta/k_x \to x$, the secondary flow  at $\theta$ becomes the secondary flow at $0$. Therefore flow patterns of the solutions are same after the phase transformation $x+\theta/k_x\to x$.
  The experimental and numerical results at the present spectral method  are displayed in Figures \ref{f1} and \ref{f2}, which show  respectively favorable agreement  with the experiment measurements  observed by  Bonderanko {\it et al.} \cite{N11}  for Reynolds number slightly over critical value (Figure \ref{f1})  and well above critical value (Figure \ref{f2}).  The analytic solution in Figure \ref{f2} is also in a good agrement with experiment measurement of Burgess \cite[Figure 1]{N12} or Figure \ref{ff3} for the secondary Kolmogorov flow pattern in a soap film  when $R$ is well above the critical Reynolds number of the laboratory experiment therein.
\begin{figure}[h]
\centering
\includegraphics[scale=0.45]{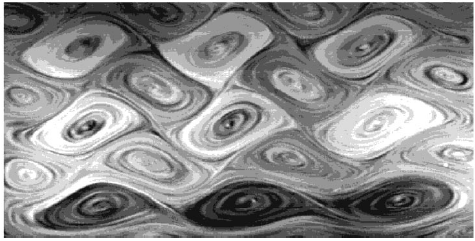}
\caption{Experimental measurement for a soap film image well above the primary instability given by Burgess {\it et al.} \cite[Figure 1]{N12}.}\label{ff3}
\end{figure}

Further numerical manipulations are  performed with respect to the increment of  high Reynolds number and different values of wave number $k_x$. Turbulence behaviours are observed for large values of $R$. In contraction to the laminar flow displayed in Figure \ref{f2} insensitive to initial data, the turbulence flow in high Reynolds number is very sensitive to the choice of initial data and time $t$. Flow patterns transited from the secondary steady states  become more and more complex as $R$ increases.  Figure 3 shows an example for four well developed flow patterns initially from  four different initial data when  $R= 40000$, $k_x=0.7$, $\lambda=20$, $N=4$ and $j=1$.

\begin{figure}[h]
\centering \vspace{0mm}
\includegraphics[scale=0.39]{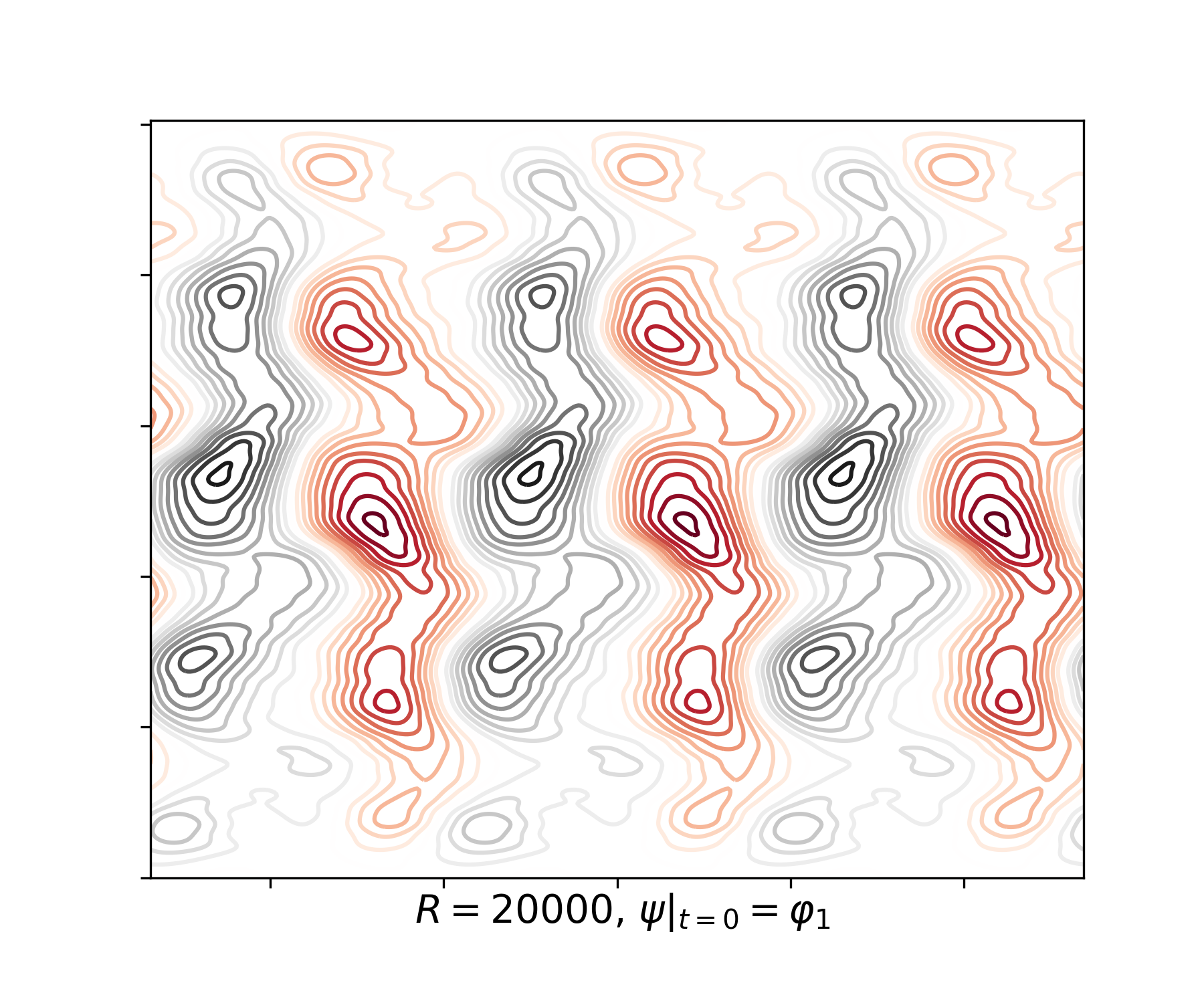}
\includegraphics[scale=0.39]{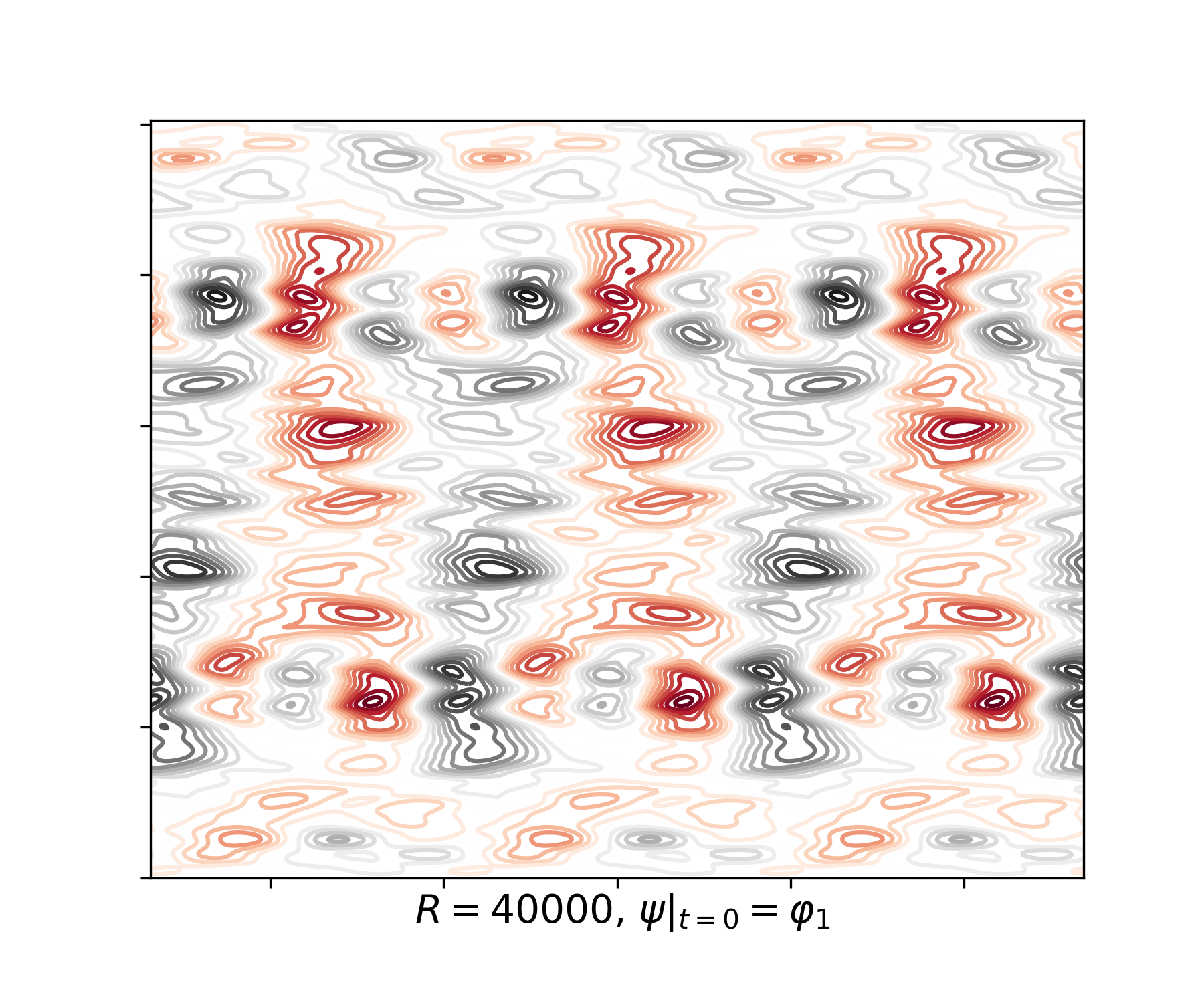}\vspace{-2mm}

\includegraphics[scale=0.39]{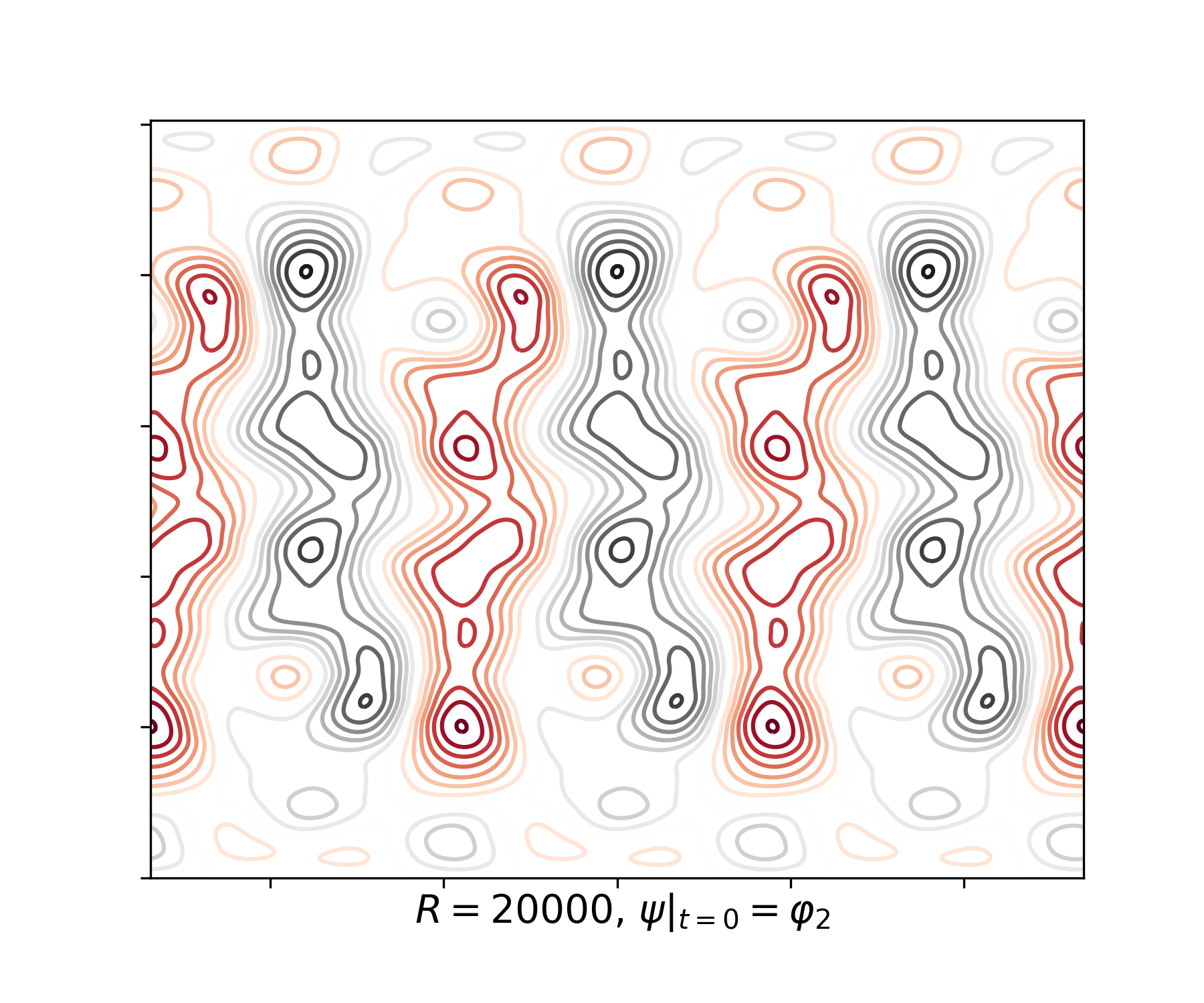}
\includegraphics[scale=0.39]{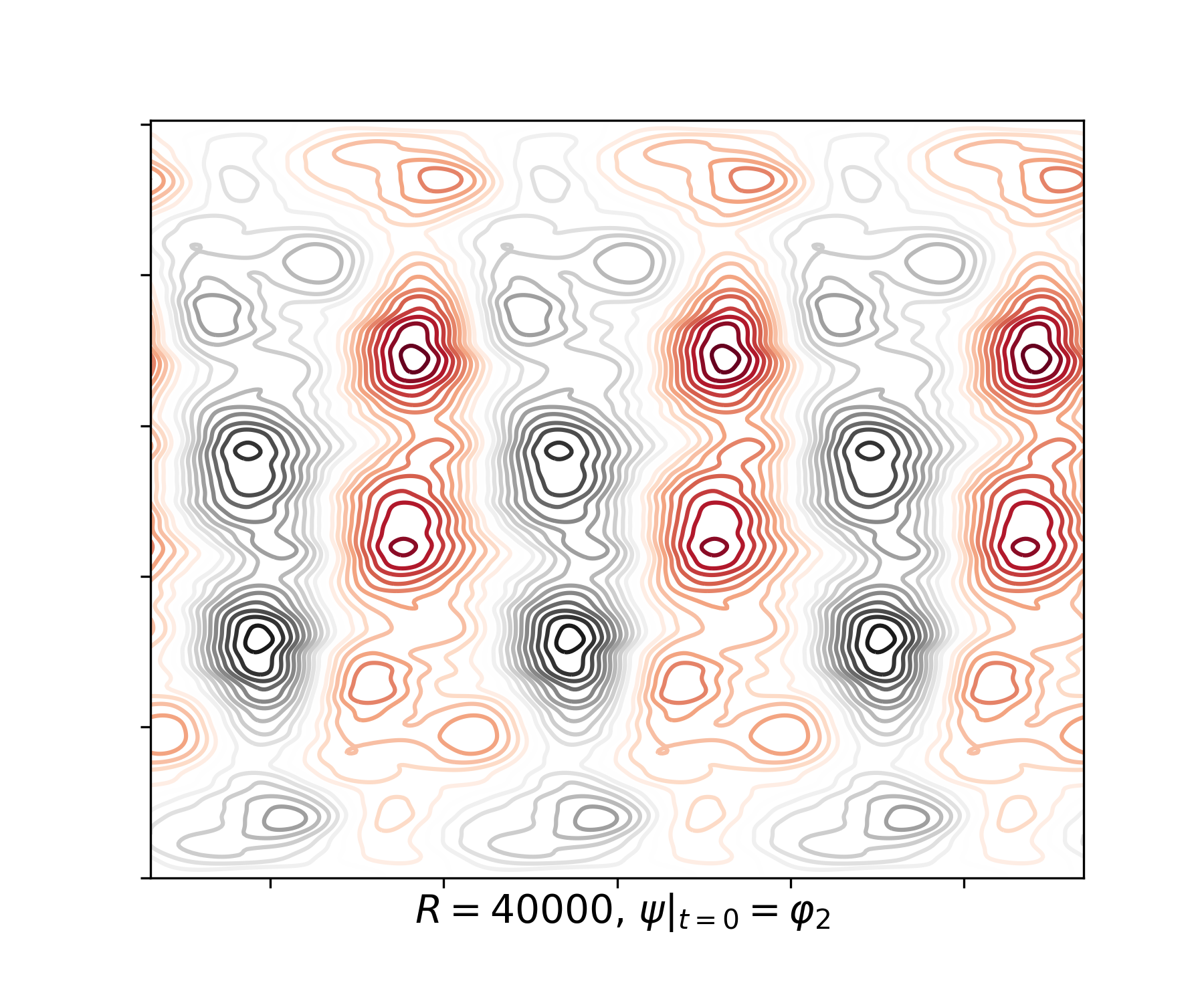}\vspace{-2mm}

\includegraphics[scale=0.39]{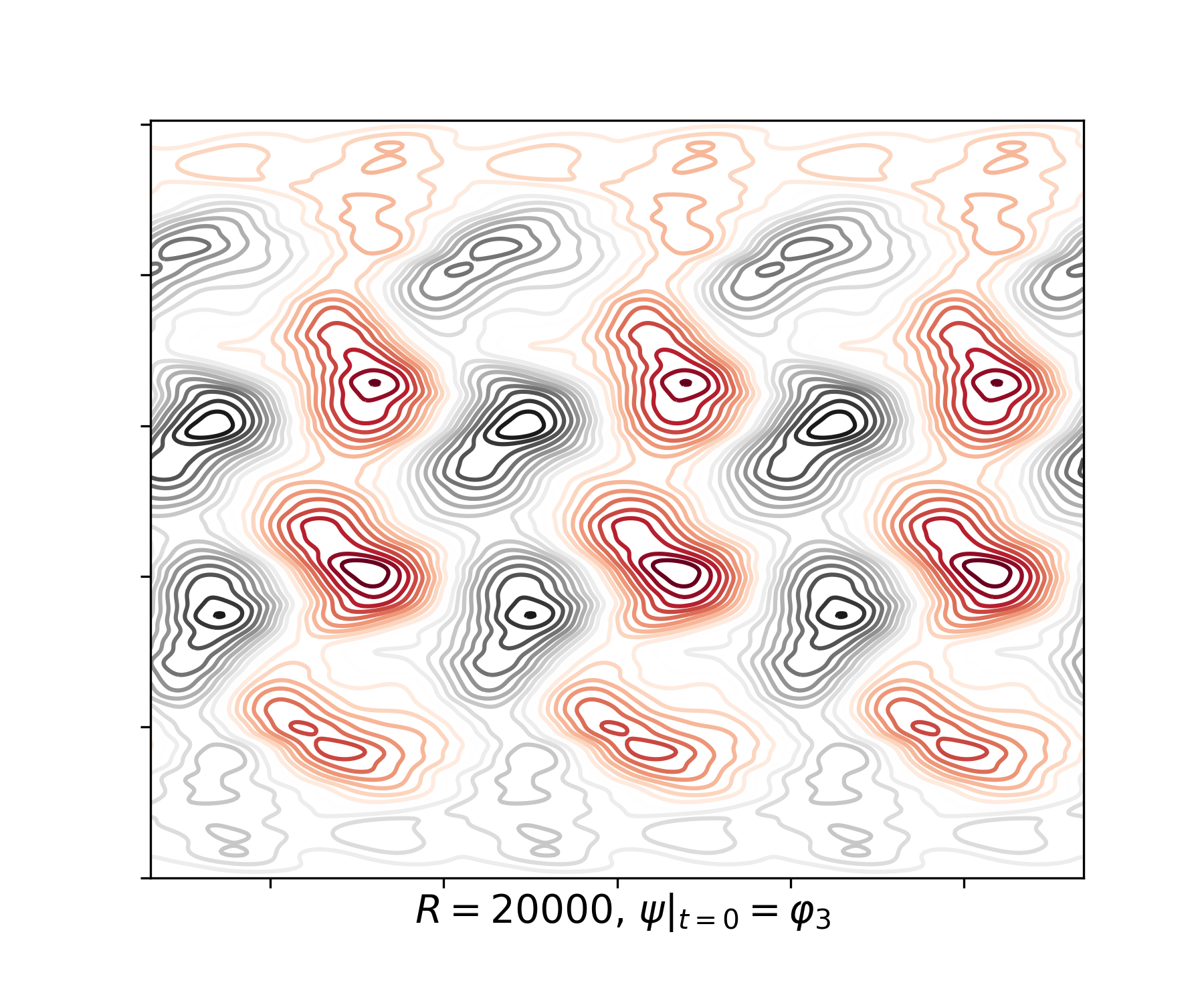}
\includegraphics[scale=0.39]{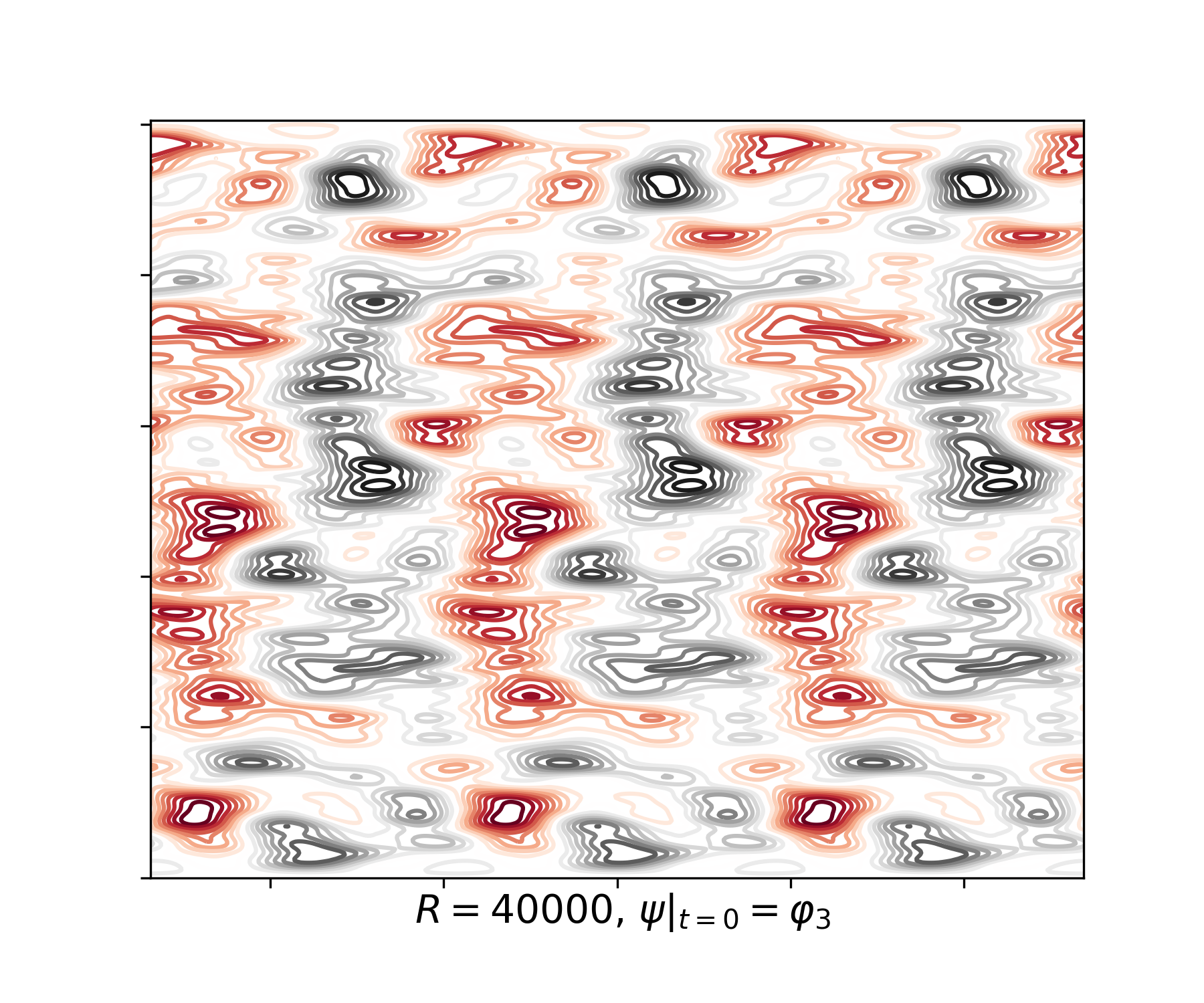}\vspace{-2mm}

\includegraphics[scale=0.39]{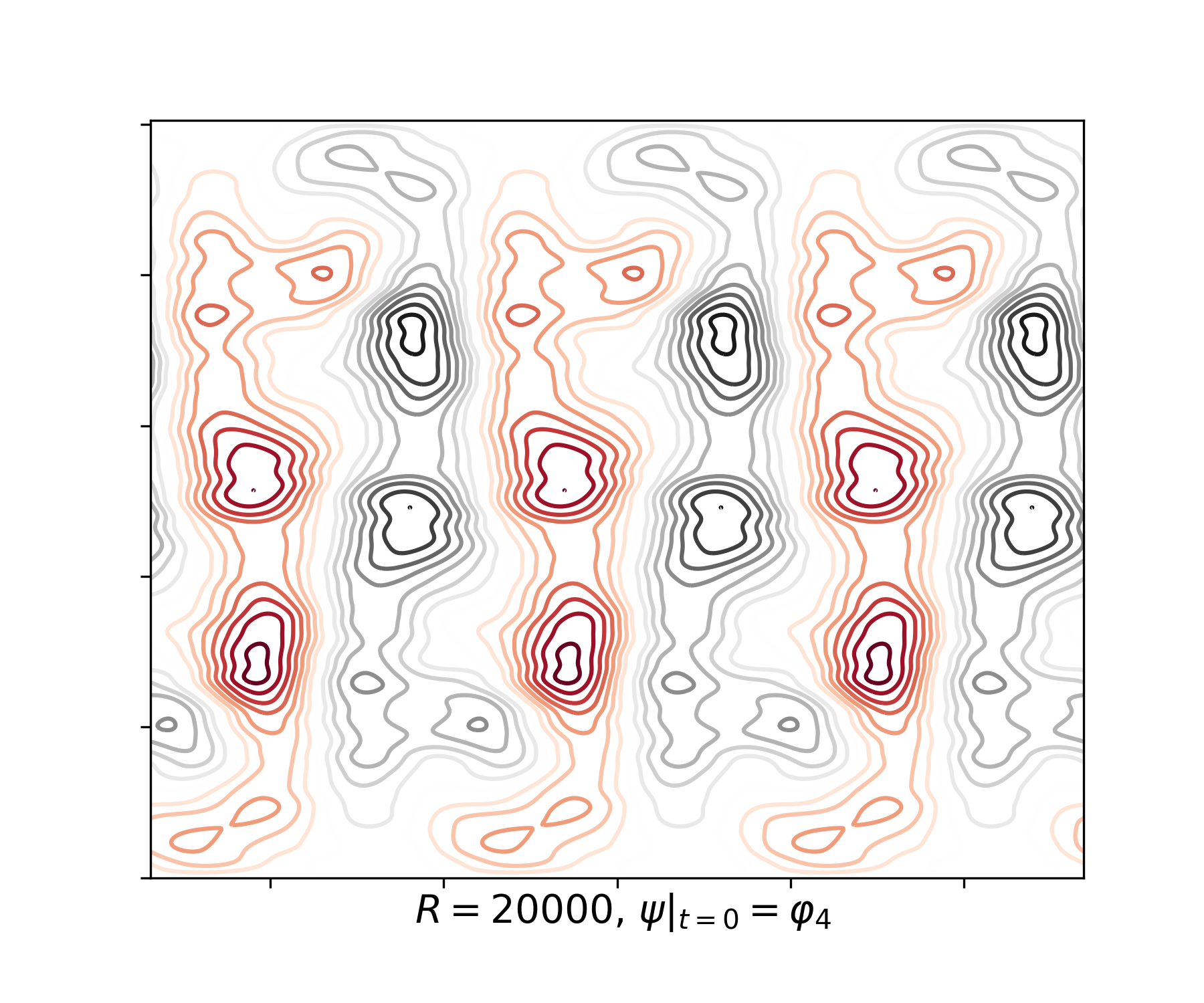}
\includegraphics[scale=0.39]{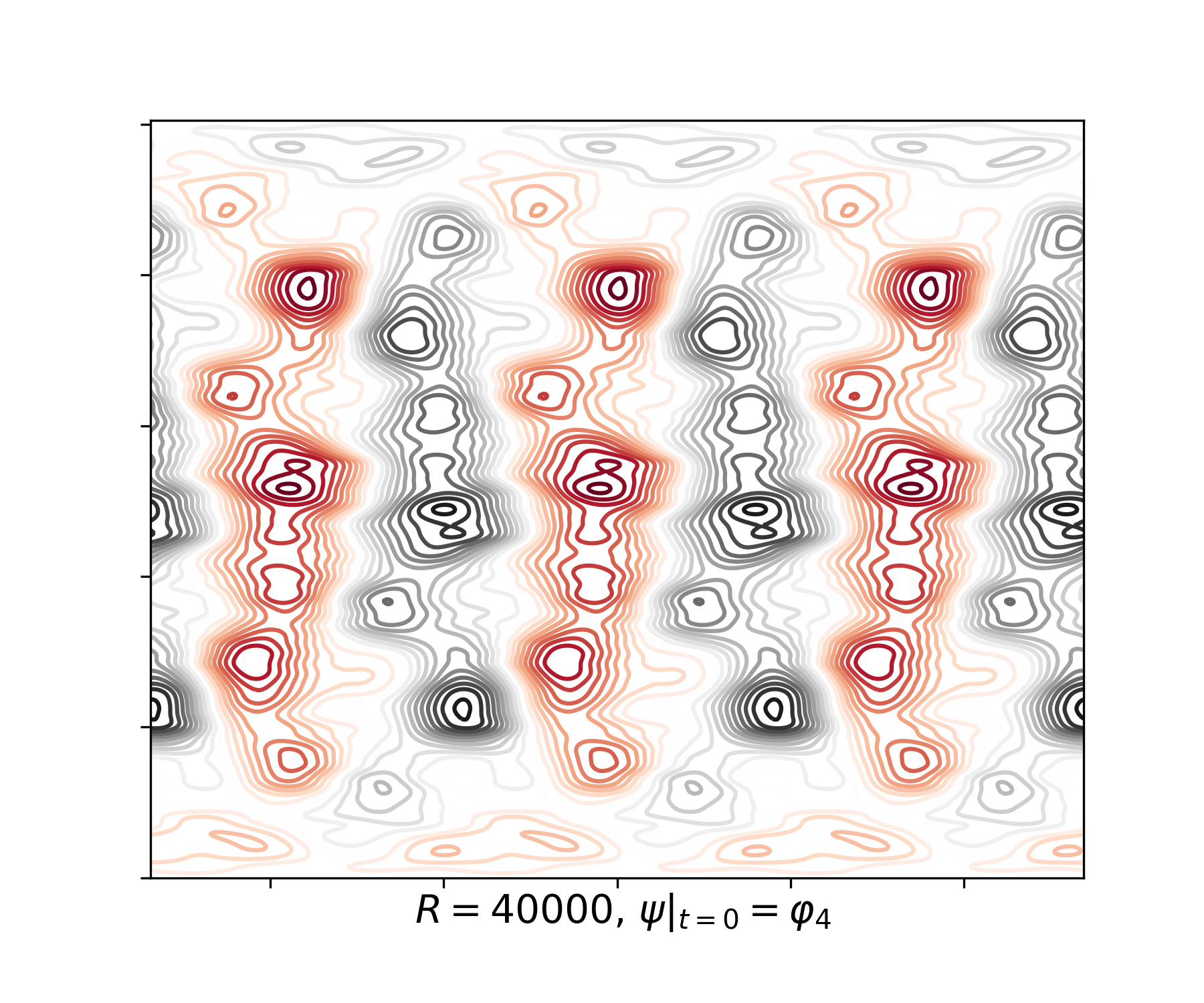}
\caption{Numerical flow patterns  of a solution $\psi$ to (\ref{new2})-(\ref{bd1})  developed  from some different  initial data $\varphi_1,...,\varphi_4$ at the  same time $t>>0$ and $R=20000, 40000$ ( $\lambda=20$, $N=4$, $j=1$, $-3\pi/k_x<x<3\pi/k_x$, $0<y<2N\pi$).\label{f3}
}
\end{figure}

\section{Conclusions}

In MHD laboratory experiments  performed by Kolesnikov \cite{N13,N14} and Bondarenko {\it et al.} \cite{N11}, an electrically conducting fluid flow  is  driven by the Lorenz force and controlled by Hartmann layer friction. This flow is governed by a two-dimensional equation  (see Bondarenko {\it et al.} \cite{N11})  and is bounded by the lateral walls of  an extended duct (see Thess \cite{N18}). The Kolmogorov flow is the basic steady-state solution of the MHD equation.

We prove rigorously that the MHD equation  admits  multiple  secondary  steady-state flows in relation to Bondarenko {\it et al.} \cite{N11}  and   confirm the finding of
Chen and Price \cite{N21} on the secondary flows defined by a simple spectral truncation scheme. The difficulty in the analysis is due to the absence of flow invariant subspace of the ducted flow  containing a single linear eigenfunction.

 The bifurcating solutions in the Fourier expansion satisfies (\ref{psi1}), which indicates   the  secondary flows being  dominated by a small number of Fourier modes. This also implies that the energy dissipation of the  MHD flow is mainly controlled  by  the Hartmann layer  friction effect.

 The theoretical secondary flow is in a good agreement with the experimental secondary flow observed by Bondarenko {\it et al.} \cite{N11} for $R=O(10^3)$. Numerical simulation is performed for further transition of the secondary flow. When $R=O(10^4)$, it is transited to well developed turbulence.

\enlargethispage{20pt}

\

\

\noindent \textbf{Acknowledgement.} This work was partially  supported by NSFC of China (11571240).

\end{document}